\documentclass[12pt,a4paper]{amsart}
\usepackage{amsmath,amssymb,color,multicol,setspace}
\usepackage{hyperref}
\usepackage[utf8,utf8x]{inputenc}
\usepackage{fullpage}

\usepackage{longtable}
\usepackage{xy,amscd}
\usepackage{epsfig}
\usepackage{rotating}
\usepackage{xspace}
\xyoption{all}

\usepackage{enumitem}

\newtheorem{propo}{Proposition}[section]
\newtheorem{corol}[propo]{Corollary}
\newtheorem{theor}[propo]{Theorem}
\newtheorem{lemma}[propo]{Lemma}

\theoremstyle{definition}
\newtheorem{defin}[propo]{Definition}
\newtheorem{examp}[propo]{Example}

\newtheorem{oppro}[propo]{Open Problem}

\theoremstyle{remark}
\newtheorem{remar}[propo]{Remark}

\newcommand{\NN }{\mathbb{N}}
\newcommand{\CC }{\mathbb{C}}

\newcommand{\ZZ }{\mathbb{Z}}

\newcommand{\id }{\mathrm{id}}
\newcommand{\ii }{\mathrm{i}}

\newcommand{\Fc }{\mathcal F}

\newcommand\frieze{frieze pattern\xspace}
\newcommand\friezes{frieze patterns\xspace}
\newcommand\cycle{quiddity cycle\xspace}
\newcommand\cycles{quiddity cycles\xspace}
\newcommand{\pls}{\oplus}

\title[A combinatorial model for tame frieze patterns]
{A combinatorial model for\\ tame frieze patterns}

\author{Michael~Cuntz}
\address{Michael Cuntz, Leibniz Universit\"at Hannover,
Institut f\"ur Algebra, Zah\-lentheorie und Diskrete Mathematik,
Fakult\"at f\"ur Mathematik und Physik,
Wel\-fengarten 1,
D-30167 Hannover, Germany}
\email{cuntz@math.uni-hannover.de}

\begin{document}


\begin{abstract}
Let $R$ be an arbitrary subset of a commutative ring. We introduce a combinatorial model for the set of tame frieze patterns with entries in $R$ based on a notion of irreducibility of frieze patterns. When $R$ is a ring, then a frieze pattern is reducible if and only if it contains an entry (not on the border) which is $1$ or $-1$. To my knowledge, this model generalizes simultaneously all previously presented models for tame frieze patterns bounded by $0$'s and $1$'s.
\end{abstract}

\maketitle

\section{Introduction}

Conway and Coxeter introduced a combinatorial model for the so-called `frieze patterns' \cite{jChC73}: their patterns, consisting entirely of positive numbers within the frieze, are in one-to-one correspondence to triangulations of a convex polygon by non-intersecting diagonals. This gives a connection between specializations of the variables of cluster algebras of type $A$ to positive integers on one side (see for example \cite{mCtH17}), and Catalan combinatorics on the other side.

Since then, many generalizations of these concepts were considered (see \cite{MG15} for a survey). In the present note, for each set $R$ of numbers, we present a combinatorial model which is associated to the set of tame frieze patterns with entries in this set $R$. Hence we generalize the above connection to arbitrary specializations of the variables in the cluster algebras of type $A$.

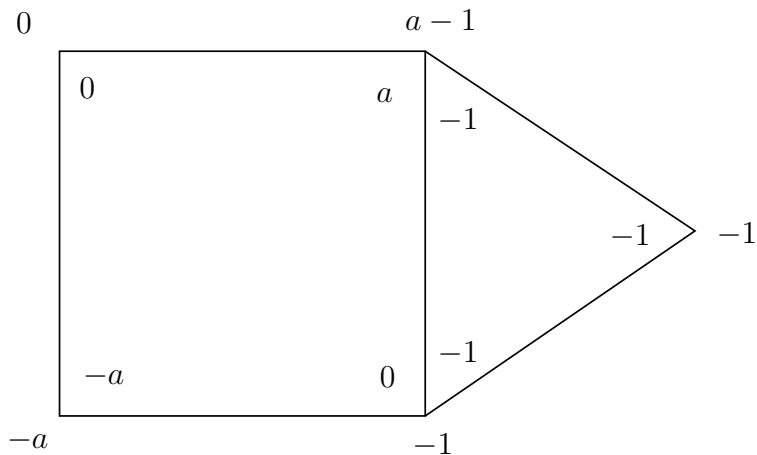
\begin{figure}[ht]
\begin{center}
\def\svgwidth{0.65\textwidth}
\begingroup%
  \makeatletter%
  \providecommand\color[2][]{%
    \errmessage{(Inkscape) Color is used for the text in Inkscape, but the package 'color.sty' is not loaded}%
    \renewcommand\color[2][]{}%
  }%
  \providecommand\transparent[1]{%
    \errmessage{(Inkscape) Transparency is used (non-zero) for the text in Inkscape, but the package 'transparent.sty' is not loaded}%
    \renewcommand\transparent[1]{}%
  }%
  \providecommand\rotatebox[2]{#2}%
  \ifx\svgwidth\undefined%
    \setlength{\unitlength}{366.82992446bp}%
    \ifx\svgscale\undefined%
      \relax%
    \else%
      \setlength{\unitlength}{\unitlength * \real{\svgscale}}%
    \fi%
  \else%
    \setlength{\unitlength}{\svgwidth}%
  \fi%
  \global\let\svgwidth\undefined%
  \global\let\svgscale\undefined%
  \makeatother%
  \begin{picture}(1,0.58544586)%
    \put(0,0){\includegraphics[width=\unitlength]{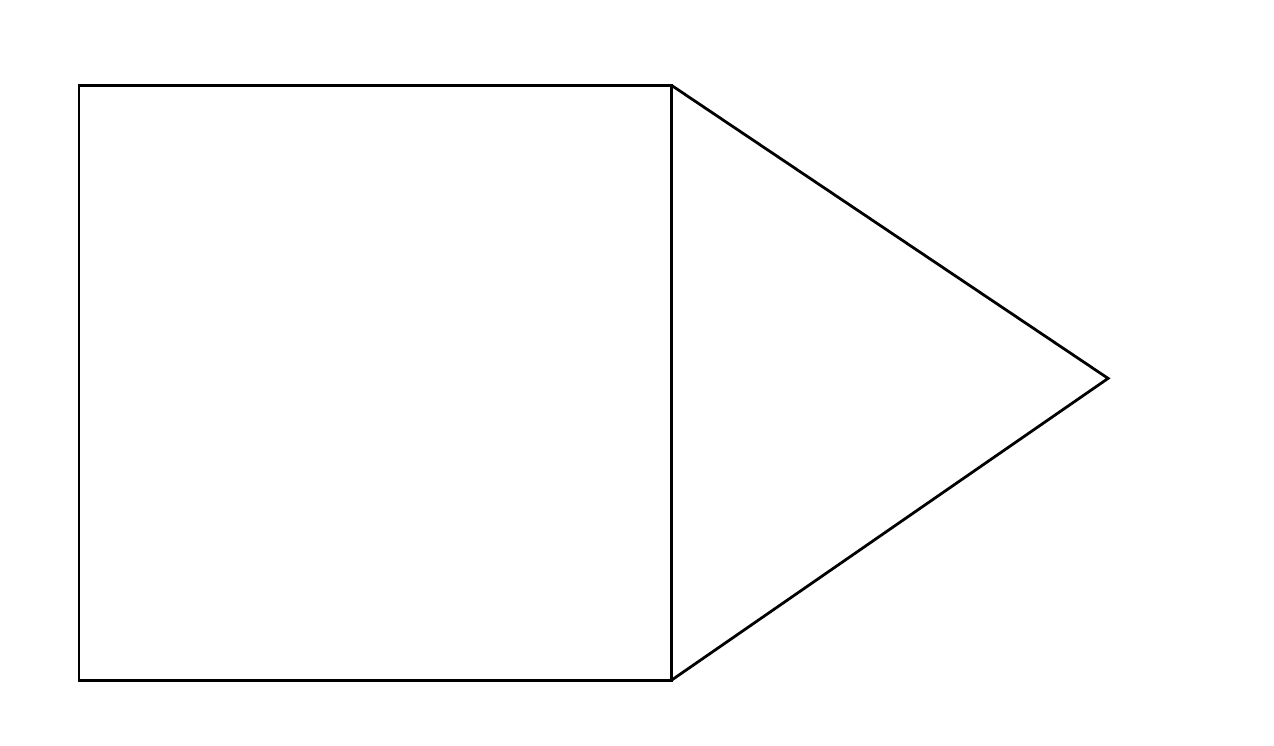}}%
    \put(0.50166532,0.54619062){\color[rgb]{0,0,0}\makebox(0,0)[lb]{\smash{$a-1$}}}%
    \put(0.46483418,0.45411272){\color[rgb]{0,0,0}\makebox(0,0)[lb]{\smash{$a$}}}%
    \put(0.0873149,0.4587166){\color[rgb]{0,0,0}\makebox(0,0)[lb]{\smash{$0$}}}%
    \put(0.0919188,0.09270698){\color[rgb]{0,0,0}\makebox(0,0)[lb]{\smash{$-a$}}}%
    \put(0.54310042,0.12033042){\color[rgb]{0,0,0}\makebox(0,0)[lb]{\smash{$-1$}}}%
    \put(0.54310042,0.41958355){\color[rgb]{0,0,0}\makebox(0,0)[lb]{\smash{$-1$}}}%
    \put(0.76178539,0.26995699){\color[rgb]{0,0,0}\makebox(0,0)[lb]{\smash{$-1$}}}%
    \put(0.89760023,0.27225896){\color[rgb]{0,0,0}\makebox(0,0)[lb]{\smash{$-1$}}}%
    \put(0.51087323,0.00293106){\color[rgb]{0,0,0}\makebox(0,0)[lb]{\smash{$-1$}}}%
    \put(-0.00476297,0.00983685){\color[rgb]{0,0,0}\makebox(0,0)[lb]{\smash{$-a$}}}%
    \put(0.00674676,0.54158666){\color[rgb]{0,0,0}\makebox(0,0)[lb]{\smash{$0$}}}%
    \put(0.46943813,0.08810316){\color[rgb]{0,0,0}\makebox(0,0)[lb]{\smash{$0$}}}%
  \end{picture}%
\endgroup%
\end{center}
\caption{$(a,0,-a,0) \pls (-1,-1,-1) = (a-1,0,-a,-1,-1)$.\label{fig0}}
\end{figure}

To this end, we introduce a notion of irreducibility of frieze patterns, Definition \ref{def:irr}. Every frieze pattern has a (not necessarily unique) decomposition into irreducible frieze patterns. In the combinatorial model, irreducible patterns become polygons that may be glued together to produce arbitrary frieze patterns (see for example Figure \ref{fig0}).

The problem of understanding this type of combinatorics for a given set $R$ thus reduces to the problem of classifying the irreducible patterns.
It turns out that a frieze pattern is reducible over a ring $R$ if and only if it contains an entry (not on the border) which is $1$ or $-1$ (see Lemma \ref{lem:irr1-1}).

\medskip
\noindent{\bf Acknowledgement:}
{I am very grateful to C.~Bessenrodt, T.~Holm, P.~J\o r\-gensen, S.~Morier-Genoud, and V.~Ovsienko for many valuable comments.}

\section{Quiddity cycles}

\begin{defin} \label{def:etamatrix}
For $c$ in a commutative ring, let
$$\eta(c) := \begin{pmatrix} c & -1 \\ 1 & 0 \end{pmatrix}.$$
\end{defin}

\begin{defin} \label{def:quiddity}
Let $R$ be a subset of a commutative ring and $\lambda\in\{\pm 1\}$.
A \emph{$\lambda$-\cycle}\footnote{Notice that the case $R=\NN_{>0}$ was also recently considered in \cite{vO17}.} over $R$ is a sequence $(c_1,\ldots,c_m)\in R^m$ satisfying
\begin{equation}\label{etaid}
\prod_{k=1}^{m} \eta(c_k) = \begin{pmatrix} \lambda & 0 \\ 0 & \lambda \end{pmatrix} = \lambda \id.
\end{equation}
A $(-1)$-\cycle is called a \emph{\cycle} for short.
\end{defin}

\begin{remar}
We agree that $m>0$ in Def.\ \ref{def:quiddity}. In fact, $m>1$ by Def.\ \ref{def:etamatrix}.
\end{remar}

\begin{examp}
Consider the commutative ring $\CC$ and $R=\CC$.
\begin{enumerate}
\item $(0,0)$ is the only $\lambda$-\cycle of length $2$.
\item $(1,1,1)$ and $(-1,-1,-1)$ are the only $\lambda$-\cycles of length $3$.
\item $(t,2/t,t,2/t)$, $t$ a unit and $(a,0,-a,0)$, $a$ arbitrary, are the only $\lambda$-\cycles of length $4$.
\end{enumerate}
\end{examp}

\begin{defin}
Let $D_n$ be the dihedral group with $2n$ elements acting on $\{1,\ldots,n\}$. If $\underline{c}=(c_1,\ldots,c_n)$ is a $\lambda$-\cycle, then we write
\[ \underline{c}^\sigma := (c_1,\ldots,c_n)^\sigma := (c_{\sigma(1)},\ldots,c_{\sigma(n)}) \]
for $\sigma\in D_n$.
\end{defin}

\begin{propo}
Let $\underline{c}=(c_1,\ldots,c_m)$ be a $\lambda$-\cycle.
Then for any $\sigma\in D_n$, the cycle $\underline{c}^\sigma$ is a $\lambda$-\cycle as well.
\end{propo}
\begin{proof}
Since the matrix $\lambda \id$ commutes with every matrix, rotating this cycle is again a $\lambda$-\cycle.
Reversing a $\lambda$-\cycle is also a $\lambda$-\cycle, see for example \cite[Prop.\ 5.3 (3)]{p-CH09d}.
\end{proof}

When thinking about a $\lambda$-\cycle $\underline{c}$, in general we do not care which element in $D_n\cdot \underline{c}$ we consider.
In the following lemma however, we have to be careful.
We introduce a \emph{sum} of $\lambda$-\cycles which is not invariant under the action of the dihedral group.
Note that a similar ``gluing'' of \friezes was already described in other papers (for instance \cite[Lemma 3.2]{tHpJ17} for \cycles in which all entries are equal or \cite{MG12} for 2-friezes).

\begin{lemma}\label{lem:comp}
Let $(a_1,\ldots,a_k)$ be a $\lambda'$-\cycle and $(b_1,\ldots,b_\ell)$ be a $\lambda''$-\cycle.
Then
\[ (a_1+b_\ell,a_2,\ldots,a_{k-1},a_k+b_1,b_2,\ldots,b_{\ell-1}) \]
is a $(-\lambda'\lambda'')$-\cycle of length $k+\ell-2$ which we call the \emph{sum}:
$$(a_1,\ldots,a_k) \pls (b_1,\ldots,b_\ell):= (a_1+b_\ell,a_2,\ldots,a_{k-1},a_k+b_1,b_2,\ldots,b_{\ell-1}).$$
\end{lemma}
\begin{proof} We use the identities $\eta(a+b)=-\eta(a)\eta(0)\eta(b)$ and $\eta(0)^2 = -\id$ (which are easy to check, see also \cite[Lemma 4.1]{mCtH17}):
\begin{eqnarray*}
&& \hspace{-22pt} \eta(a_1+b_\ell)\eta(a_2)\cdots\eta(a_{k-1})\eta(a_k+b_1)\eta(b_2)\cdots\eta(b_{\ell-1}) \\
&=& \eta(b_\ell)\eta(0)\eta(a_1)\eta(a_2)\cdots\eta(a_{k-1})\eta(a_k)\eta(0)\eta(b_1)\eta(b_2)\cdots\eta(b_{\ell-1}) \\
&=& \lambda' \eta(b_\ell)\eta(0)\eta(0)\eta(b_1)\eta(b_2)\cdots\eta(b_{\ell-1}) \\
&=& -\lambda' \eta(b_\ell)\eta(b_1)\eta(b_2)\cdots\eta(b_{\ell-1}) = -\lambda'\lambda'' \id. \qedhere
\end{eqnarray*}
\end{proof}

\begin{examp}
\begin{enumerate}
\item If $(a_1,\ldots,a_m)$ is a \cycle, then $$(a_1,\ldots,a_m)\pls (0,0) = (a_1,\ldots,a_m).$$
\item For $a\in\CC$, $(a,0,-a,0)$ and $(-1,-1,-1)$ are $1$-\cycles, their sum is\\ $(a-1,0,-a,-1,-1)$ and is a \cycle (see also Fig.\ \ref{fig0}).
\end{enumerate}
\end{examp}

The following are the central notions of \emph{reducibility} and \emph{irreducibility} of \cycles mentioned in the introduction.

\begin{defin}\label{def:irr}
Let $R$ be a subset of a commutative ring.
A $\lambda$-\cycle $(c_1,\ldots,c_m)\in R^m$, $m>2$ is called \emph{reducible over $R$} if there exist
a $\lambda'$-\cycle $(a_1,\ldots,a_k)\in R^k$, a $\lambda''$-\cycle $(b_1,\ldots,b_\ell)\in R^\ell$, and $\sigma\in D_m$ such that
$\lambda = -\lambda'\lambda''$, $k,\ell>2$ and
\begin{eqnarray*}
(c_1,\ldots,c_m)^\sigma &=& (a_1+b_\ell,a_2,\ldots,a_{k-1},a_k+b_1,b_2,\ldots,b_{\ell-1})\\
&=& (a_1,\ldots,a_k)\pls (b_1,\ldots,b_\ell).
\end{eqnarray*}
A $\lambda$-\cycle of length $m>2$ is called \emph{irreducible over $R$} if it is not reducible.
\end{defin}

\begin{remar}
There is no need to consider the cycle of length $m<3$ (which is $(0,0)$) in Definition \ref{def:irr}.
\end{remar}

\begin{defin}
Consider a $\lambda$-\cycle $\underline{c}=(c_1,\ldots,c_m)$ and define $c_k$ for all $k\in\ZZ$ by repeating $\underline{c}$ periodically.
For $i,j\in\ZZ$ let
\[ x_{i,j}:= \left(\prod_{k=i}^{j-2} \eta(c_k)\right)_{1,1} \quad \text{if } i\le j-2, \]
$x_{i,i+1}:=1$, and $x_{i,i}:=0$.
Notice that $x_{i,i+2}=c_i$.\\
Then we call the array $\Fc=(x_{i,j})_{i\le j\le i+m}$ the \emph{\frieze} of $\underline{c}$. The \emph{entries} of the \frieze of $\underline{c}$ are the numbers $x_{i,j}$ with $i+2\le j \le i+m-2$.\\
We say that the \frieze of $\underline{c}$ is \emph{reducible} resp.\ \emph{irreducible} if $\underline{c}$ is \emph{reducible} resp.\ \emph{irreducible}.
\end{defin}

\begin{remar}
(a) If $\underline{c}$ is a \cycle, then we obtain
what we usually call the frieze pattern. In fact, in this way we exactly obtain all \emph{tame} frieze patterns, i.e.\ those for which every adjacent $3\times 3$ determinant is zero (see for example \cite[Prop.\ 2.4]{mCtH17}).\\
Starting with a $1$-\cycle, one obtains a \frieze with $1$'s on one border and $-1$'s on the other border, i.e.\ $x_{i,i+m-1}=-1$ for all $i$.

(b) The entries $x_{i,j}$ of a \frieze are specialized cluster variables of a cluster algebra of Dynkin type $A$ (see for example \cite[Section 5]{mCtH17}).

(c) Notice that if $\underline{c}$ is a $\lambda$-\cycle over $R$, then its \frieze may have entries which are not in $R$.
It is an interesting question to determine the set of entries of \friezes of $\lambda$-\cycles for a fixed set $R$.
For example, if $R$ is a ring then all entries in the \friezes are in $R$.
\end{remar}

The following lemma explains the appearance of $1$'s and $-1$'s in friezes. Some similar statement is contained implicitly for the case $R=\NN_{>0}$ in \cite[Cor.\ 1.11]{sMG14} for Coxeter friezes.

\begin{lemma}\label{lem:irr1-1}
Let $R$ be a commutative ring.
A $\lambda$-\cycle is reducible over $R$ if and only if the corresponding tame frieze pattern contains an entry $1$ or $-1$.
\end{lemma}
\begin{proof}
Reducibility requires that the length $m$ of the cycle is at least $4$; since there are no entries in a frieze pattern with $\lambda$-\cycle of length less than $4$, we may assume $m\ge 4$.\\
Assume first the existence of an entry $\varepsilon=\pm 1$, i.e.\ without loss of generality (rotating the cycle if necessary) there are $i,j\in\{1,\ldots,m\}$ with $i<j-1$, $j-i<m-1$ and
$M_{1,1}=\varepsilon$ for $M=\prod_{k=i}^{j-2} \eta(c_k)$. Since $\det(M)=1$, with $a:=\varepsilon M_{2,1}$, $b:=-\varepsilon M_{1,2}$ we have
\[ M = \begin{pmatrix} \varepsilon & -\varepsilon b \\ \varepsilon a & -\varepsilon ab+\varepsilon \end{pmatrix}
= -\varepsilon \begin{pmatrix} -1 & b \\ -a & ab-1 \end{pmatrix}
 = -\varepsilon \eta(a)^{-1}\eta(b)^{-1}. \]
We obtain
\[ \eta(a) \left(\prod_{k=i}^{j-2} \eta(c_k)\right) \eta(b) = -\varepsilon \id, \]
so $(a,c_i,\ldots,c_{j-2},b)$ is a $(-\varepsilon)$-\cycle. It follows that
\[ (c_{j-1}-b,c_j,\ldots,c_m,c_1,\ldots,c_{i-2},c_{i-1}-a) \]
is a $(\lambda\varepsilon)$-\cycle of length $m-j+i+1\ge 3$ since $j-i<m-1$; thus the cycle is reducible.\\
For the converse, assume that we have a decomposition into a sum as above. But then $\left(\prod_{k=i}^{j-2} \eta(c_k)\right)_{1,1}\in\{\pm 1\}$ for some $i,j$ with $i<j-2$ which gives an entry $\pm 1$ in the pattern.
\end{proof}

\section{Examples of subsets}

Some classifications of irreducible $\lambda$-\cycles are already known.
For example, every \cycle over $\NN_{>0}$ contains a $1$. Thus any \cycle over $\NN_{>0}$ of length greater than $3$ has a summand $(1,1,1)$ (cf.\ {\cite{jChC73}}), although the other summand only has positive entries if the original frieze pattern has no entry zero. In general:
\begin{theor}
The only irreducible $\lambda$-\cycles over $\ZZ_{\ge 0}$ are $(0,0,0,0)$ and $(1,1,1)$.
\end{theor}
\begin{proof}
Let $\underline{c}=(c_1,\ldots,c_m)\in \ZZ_{\ge 0}^m$, $m>2$ be a $\lambda$-\cycle.

If $c_i>0$ for all $i$ then by \cite[Cor.\ 3.3]{mCtH17} there exists a $j\in\{1,\ldots,m\}$ with $c_j=1$, without loss of generality $j=2$. But then
$\underline{c}=(1,1,1)\oplus \underline{c}'$ where $\underline{c}'=(c_3-1,c_4,\ldots,c_m,c_1-1)\in\ZZ_{\ge 0}^{m-1}$.

Otherwise there are zeros in $\underline{c}$. If $\underline{c}$ contains two adjacent zeros, say $c_2=c_3=0$ then $\underline{c}=(0,0,0,0)\oplus  \underline{c}'$ where $\underline{c}'=(c_4,\ldots,c_m,c_1)\in\ZZ_{\ge 0}^{m-2}$.

The last case is when there are zeros, but none of them has an adjacent zero. Notice first that since $\eta(a)\eta(0)\eta(b)=-\eta(a+b)$ for all $a,b$ (cf.\ \cite[Lem.\ 4.1]{mCtH17}), if $(c_1,0,c_3,\ldots,c_m)$ is a $\lambda$-\cycle, then $(c_1+c_3,\ldots,c_m)$ is a $(-\lambda)$-\cycle. Applying this transformation to all zeros simultaneously yields a $\lambda$-\cycle $\underline{c}''$ in which only the entries coming from $\underline{c}$ which were not adjacent to a zero may be $\le 1$.
But by \cite[Cor.\ 3.3]{mCtH17} there exists an entry $\le 1$ in $\underline{c}''$, so we find a $1$ in $\underline{c}$ which has nonzero adjacent entries, hence $\underline{c}^\sigma=(1,1,1)\oplus \underline{c}'$ for some $\underline{c}'\in\ZZ_{\ge 0}^{m-1}$ and $\sigma\in D_m$ as in the first case.
\end{proof}

If we allow entries in the set of all integers, the situation is slightly more complicated:
\begin{theor}[{\cite[Thm.\ 6.2]{mCtH17}}]
The set of irreducible $\lambda$-\cycles over $\ZZ$ is
\[ \{ (1,1,1), (-1,-1,-1), (a,0,-a,0), (0,a,0,-a) \mid a\in \ZZ\setminus\{\pm 1\}\}. \]
\end{theor}

\begin{propo}
Let $k\in\NN_{>0}$ and $\ii=\sqrt{-1}$. Then
\[ \underline{c} = (2\ii,-\ii+1, \underbrace{2,\ldots,2}_{2k\text{-times}}, \ii+1,
    -2\ii,\ii-1, \underbrace{-2,\ldots,-2}_{2k\text{-times}}, -\ii-1) \]
is an irreducible \cycle over $\ZZ[\ii]$.
\end{propo}
\begin{proof}
Notice first that
\[ \eta(2)^\ell = \begin{pmatrix} \ell+1 & -\ell \\ \ell & 1-\ell \end{pmatrix}, \quad
\eta(-2)^\ell = (-1)^\ell \begin{pmatrix} \ell+1 & \ell \\ -\ell & 1-\ell \end{pmatrix} \]
for $\ell\in\NN_{>0}$. It is then easy to check that $\underline{c}$ is a \cycle.
Further, using the same identities we can compute each type of entry in the frieze pattern.
We compute $x_{1,2k+5}$ as an example:
\[ \prod_{i=1}^{2k+3} \eta(c_i) =
\eta(2\ii)\eta(-\ii+1) \eta(2)^{2k} \eta(\ii+1)
= \begin{pmatrix} 2\ii k + \ii - 1 & -2k - 2\ii - 1 \\  2k + 1 &  2\ii k + \ii - 1 \end{pmatrix}
\]
and thus $x_{1,2k+5} = 2\ii k + \ii - 1$.
It turns out that none of them is $\pm 1$ and hence it is irreducible by Lemma \ref{lem:irr1-1}.
\end{proof}

This immediately yields:

\begin{corol}
There are infinitely many irreducible $\lambda$-\cycles over the Gaussian numbers $\ZZ[\ii]$.
\end{corol}

\section{Combinatorial model}

Let $(a_1,\ldots,a_k)$ be a $\lambda'$-\cycle and $(b_1,\ldots,b_\ell)$ be a $\lambda''$-\cycle.
If we represent these two cycles as polygons, then gluing them together yields a larger polygon representing their sum, see Figure \ref{fig1}.

\begin{figure}[ht]
\begin{center}
\def\svgwidth{0.65\textwidth}
\begingroup%
  \makeatletter%
  \providecommand\color[2][]{%
    \errmessage{(Inkscape) Color is used for the text in Inkscape, but the package 'color.sty' is not loaded}%
    \renewcommand\color[2][]{}%
  }%
  \providecommand\transparent[1]{%
    \errmessage{(Inkscape) Transparency is used (non-zero) for the text in Inkscape, but the package 'transparent.sty' is not loaded}%
    \renewcommand\transparent[1]{}%
  }%
  \providecommand\rotatebox[2]{#2}%
  \ifx\svgwidth\undefined%
    \setlength{\unitlength}{514.93406434bp}%
    \ifx\svgscale\undefined%
      \relax%
    \else%
      \setlength{\unitlength}{\unitlength * \real{\svgscale}}%
    \fi%
  \else%
    \setlength{\unitlength}{\svgwidth}%
  \fi%
  \global\let\svgwidth\undefined%
  \global\let\svgscale\undefined%
  \makeatother%
  \begin{picture}(1,0.50654146)%
    \put(0,0){\includegraphics[width=\unitlength]{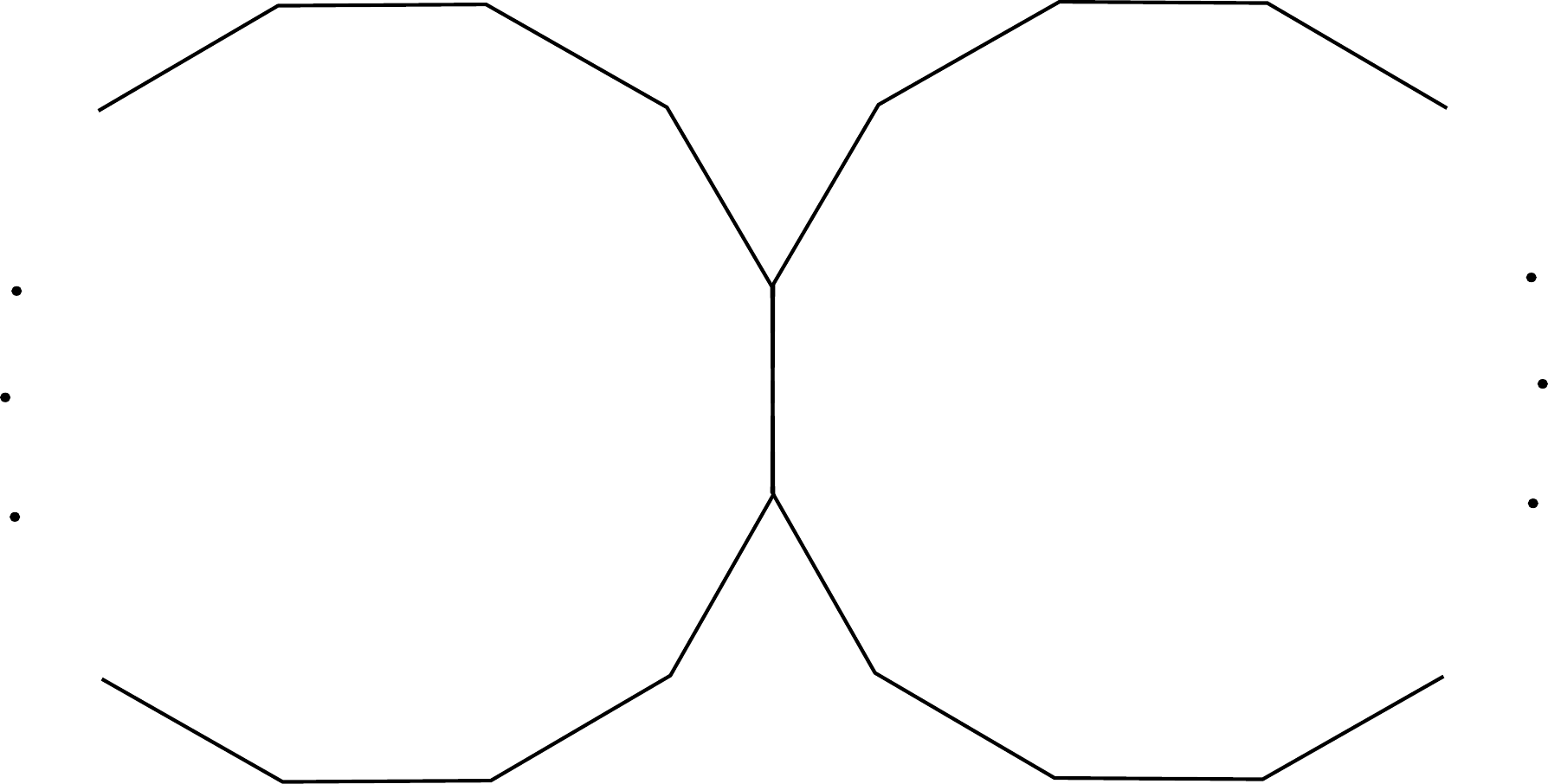}}%
    \put(0.42341778,0.31221657){\color[rgb]{0,0,0}\makebox(0,0)[lb]{\smash{$a_1$}}}%
    \put(0.36730182,0.41927309){\color[rgb]{0,0,0}\makebox(0,0)[lb]{\smash{$a_2$}}}%
    \put(0.26244896,0.47169955){\color[rgb]{0,0,0}\makebox(0,0)[lb]{\smash{$a_3$}}}%
    \put(0.42741518,0.18578825){\color[rgb]{0,0,0}\makebox(0,0)[lb]{\smash{$a_k$}}}%
    \put(0.52176254,0.31047408){\color[rgb]{0,0,0}\makebox(0,0)[lb]{\smash{$b_\ell$}}}%
    \put(0.52627237,0.18250853){\color[rgb]{0,0,0}\makebox(0,0)[lb]{\smash{$b_1$}}}%
    \put(0.32896991,0.08595751){\color[rgb]{0,0,0}\makebox(0,0)[lb]{\smash{$a_{k-1}$}}}%
    \put(0.57024289,0.07765567){\color[rgb]{0,0,0}\makebox(0,0)[lb]{\smash{$b_2$}}}%
    \put(0.68016931,0.01508224){\color[rgb]{0,0,0}\makebox(0,0)[lb]{\smash{$b_3$}}}%
    \put(0.57080662,0.40236127){\color[rgb]{0,0,0}\makebox(0,0)[lb]{\smash{$b_{\ell-1}$}}}%
  \end{picture}%
\endgroup%
\end{center}
\caption{$(a_1,\ldots,a_k) \pls (b_1,\ldots,b_\ell)$.\label{fig1}}
\end{figure}
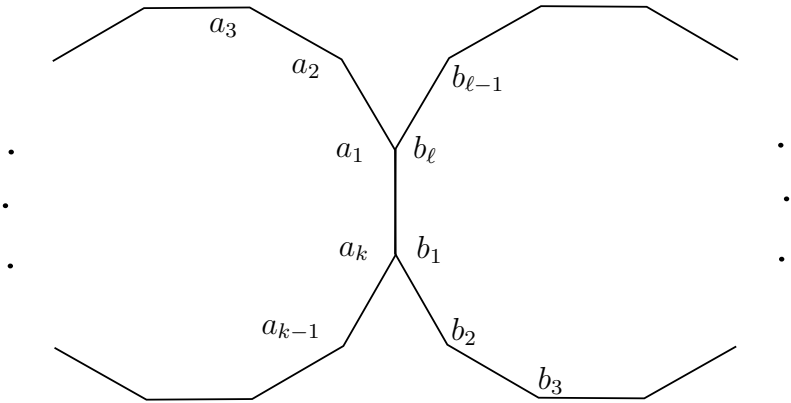

We see the sum
$$(a_1,\ldots,a_k) \pls (b_1,\ldots,b_\ell) = (a_1+b_\ell,a_2,\ldots,a_{k-1},a_k+b_1,b_2,\ldots,b_{\ell-1})$$
in the new polygon when adding the entries at the vertices which are glued together.

Hence the decomposition of a $\lambda$-\cycle into a sum of irreducible ones translates in a natural way into a polygon decomposed into building blocks which correspond to some irreducible summands.

Since the only irreducible $\lambda$-\cycles for $R=\NN_{\ge 0}$ are $(0,0,0,0)$ and $(1,1,1)$,
in this special case we recover the Catalan combinatorics originally proposed by Conway and Coxeter.

More precisely:
It is easy to prove that if the frieze pattern of a $\lambda$-\cycle $\underline{c}$ for $R=\NN_{>0}$ has only positive entries, then $\underline{c}$ is a sum of \cycles $(1,1,1)$.
The $(0,0,0,0)$-polygons are the parts that glue classical Conway-Coxeter friezes together; they produce zeros within the corresponding frieze pattern.

We close this note with a somewhat vague task:

\begin{oppro}
Classify irreducible \cycles for some of the most interesting sets $R\subseteq \CC$.
\end{oppro}

\bibliographystyle{amsalpha}

\def\cprime{$'$}
\providecommand{\bysame}{\leavevmode\hbox to3em{\hrulefill}\thinspace}
\providecommand{\MR}{\relax\ifhmode\unskip\space\fi MR }
\providecommand{\MRhref}[2]{%
  \href{http://www.ams.org/mathscinet-getitem?mr=#1}{#2}
}
\providecommand{\href}[2]{#2}

\end{document}